\documentclass[11pt]{article}
\usepackage{amsmath}
\usepackage{amsthm}
\usepackage{amssymb}
\usepackage{dsfont}
\usepackage{pgfplots}

\newtheorem{theo}{Theorem}[section]
\newtheorem{thm}{Theorem}
\newtheorem{definition}{Definition}[section]
\newtheorem{prop}[theo]{Proposition}
\newtheorem{lemma}[theo]{Lemma}
\newtheorem{coro}[theo]{Corollary}

\newtheorem{remark}[theo]{Remark}

\begin{document}
\date{}

\title{
On a random model of forgetting \\
}

\author{Noga Alon
\thanks
{Department of Mathematics, Princeton University, Princeton,
NJ 08544, USA and Schools of Mathematics and Computer Science,
Tel Aviv University, Tel Aviv, Israel.
Email: {\tt nalon@math.princeton.edu}.
Research supported in part by
NSF grant DMS-1855464, ISF grant 281/17,
BSF grant 2018267
and the Simons Foundation.}
\and Dor Elboim
\thanks{Department of Mathematics, Princeton University,
Princeton, NJ 08544, USA. Email:
	{\tt delboim@princeton.edu}.}
\and  Allan Sly
\thanks{Department of Mathematics, Princeton University,
Princeton, NJ 08544, USA. Email:
	{\tt allansly@princeton.edu}. Research supported in part by  NSF grant DMS-1855527, a Simons Investigator grant and a MacArthur Fellowship.}
}

\maketitle
\begin{abstract}
Georgiou, Katkov and Tsodyks considered the following random process.
Let $x_1,x_2,\ldots $ be an infinite sequence of independent, identically distributed,
uniform random points in $[0,1]$.
Starting with $S=\{0\}$, the
elements $x_k$ join $S$ one by one, in order. When an entering
element is larger than the current minimum element of $S$, this minimum
leaves $S$. Let $S(1,n)$ denote the content of $S$ after the first $n$
elements $x_k$ join.
Simulations suggest that the size $|S(1,n)|$ of $S$
at time $n$ is typically close to $n/e$. Here we first give a rigorous proof
that this is indeed the case, and that in fact the symmetric
difference of $S(1,n)$ and the set 
$\{x_k\ge 1-1/e: 1 \leq k \leq n \}$ is of size at
most $\tilde{O}(\sqrt n)$ with high probability.  Our main result
is a more accurate description of the process implying, in
particular, that as $n$ tends to infinity
$ n^{-1/2}\big( |S(1,n)|-n/e \big) $ converges
to a normal random variable with variance $3e^{-2}-e^{-1}$. We further show that the dynamics of the symmetric difference of $S(1,n)$ and the set 
$\{x_k\ge 1-1/e: 1 \leq k \leq n \}$ converges with proper scaling to a three dimensional Bessel process.
\end{abstract}

\section{Introduction}

The following random process, which we denote by $P$, is considered by 
Georgiou, Katkov and Tsodyks in \cite{GKT}
motivated by the study of a possible simplified model for
the process of forgetting.
Let $x_1,x_2, \ldots $ be independent, 
identically distributed, uniform
random points in $[0,1]$. 
Starting with the set $S(1,0)=\{0\}=\{x_0\}$ the elements
$x_k$ enter the set one by one, in order, and if when 
$x_k$ enters $S(1,k-1)$ it is larger than 
the minimum member $\min(S(1,k-1))$ of $S(1,k-1)$, this minimum leaves it. Let us note that in this process, the fact that $x_i$ are uniformly distributed is not crucial. Indeed, it only matters that the ordering of the arriving elements is a uniform random ordering and therefore the uniform distribution can be replaced with any other non-atomic distribution.

The set $S(1,n)$ will be called the memory. We let $P(n)$ be the finite process that stops at time $n$. Simulations indicate that the expected size of the memory 
$S(1,n)$ at time $n$ is $(1+o(1))n/e$ where $e$ is the basis of the natural logarithm,
and that in fact with high probability (that is, with probability 
approaching $1$ as $n$ tends to infinity) the size at the end is very close
to the expectation. Tsodyks \cite{Ts} suggested the problem of proving 
this behavior rigorously. Our first result here is a proof that
this is indeed the case, as stated in the following theorem. Recall that for two function $f,g$ the notation $f(n)=\tilde O (g(n))$ means that for all large $n$ one has $f(n)\le C(\log n)^C g(n)$.

\begin{thm}
\label{t11}
In the process $P(n)$ the size of the memory 
$S(1,n)$ at time $n$ is, with high probability,
$(1+o(1))n/e$. Moreover, with high 
probability, the  symmetric difference between the final content
of the memory and the set $\{x_k \ge 1-1/e: k \leq n\}$
is of size at most $\tilde{O}(\sqrt {n})$. 
\end{thm}
A result similar to Theorem~\ref{t11} (without any quantitative estimates) has been proved
independently by Friedgut and Kozma \cite{FK} using different ideas.

\begin{figure}[htp]
    \centering
    \includegraphics[width=14.5cm]{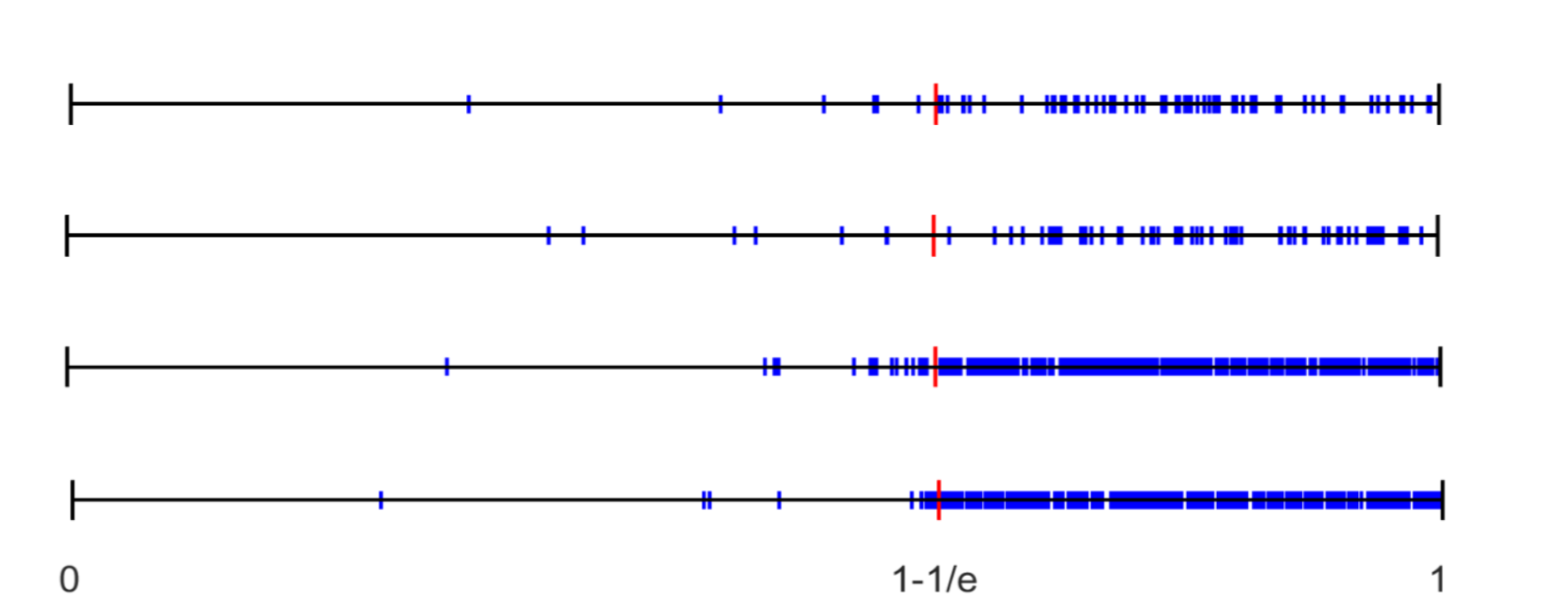}
    \caption{Four independent samples of the process obtained from a computer simulation. The blue points are the content of the memory set after $n=200$ steps in the first two pictures and after $n=1000$ steps in the last two. In the main theorems we analyze the evolution of the memory set and determine how close it is to the set of uniform variables that are larger than $1-1/e$.}
    \label{fig:1}
\end{figure}

Our main result provides a more accurate description of the process
$P$ as stated in the following theorems. To state the next theorem we let $S(1,n)$ be the content of the list at time $n$ and let $s(1,n):=|S(1,n)|$ be the size of the list.  

\begin{thm}\label{theo:size}
Let $B_t$ be a standard Brownian motion. We have that 
\begin{equation*}
    \Big\{ \frac{s(1,tn)-tn/e}{\sqrt{n}} \Big\} _{t>0} \overset{d}{\longrightarrow } \Big\{ \frac{\sqrt{3-e}}{e}\cdot  B_t \Big\}_{t>0} ,\quad n\to \infty .
\end{equation*}
\end{thm}

Our next theorem describes the final set of points obtained in this process. In particular, it strengthens the estimate in Theorem~\ref{t11} by removing the poly-logarithmic factor in the bound on the size of the symmetric difference between $S(1,n)$ and the set  $\{x_k \ge z_0: k\le n\}$ where $z_0:=1-1/e$. The theorem also provides the limiting distribution of this difference.
To state the theorem we define 
\begin{equation*}
       L_n:=\big|S(1,n) \setminus \big\{ x_k\ge z_0 :k\le n  \big\} \big| \quad \text{and} \quad R_n:=\big|\big\{  x_k\ge z_0 : k\le n  \big\} \setminus S(1,n) \big| . 
\end{equation*}

\begin{thm}\label{theo:12}
As $n$ tends to infinity we have
\begin{equation*}
\Big\{ \Big(  \frac{L_{tn}}{\sqrt{n}} \, , \, \frac{R_{tn}}{\sqrt{n}} \Big)  \Big\}_{t>0} \overset{d}{\longrightarrow } \big\{ \sqrt{2}e^{-1} \big( M_t-B_t, M_t \big) \big\} _{t>0},
\end{equation*}
where $B_t$ is a standard Brownian motion and $M_t:=\sup _{s\le t} B_s$ is the maximum process. In particular, the evolution of the size of the symmetric difference is given by
\begin{equation}\label{eq:936}
\Big\{ \frac{L_{tn}+R_{tn}}{\sqrt{n}}   \Big\}_{t>0} \overset{d}{\longrightarrow } \big\{ \sqrt{2}e^{-1} X_t \big\} _{t>0},
\end{equation}
where $X_t$ is a three dimensional Bessel process. It follows that the differences and the symmetric difference satisfy the following limit laws 
\begin{equation}\label{eq:LR}
    \frac{L_n}{\sqrt{n}} \overset{d}{\longrightarrow } \sqrt{2}e^{-1}|N|, \quad  \frac{R_n}{\sqrt{n}} \overset{d}{\longrightarrow } \sqrt{2}e^{-1}|N| \quad \text{and} \quad \frac{L_n+R_n}{\sqrt{n}} \overset{d}{\longrightarrow } X, 
\end{equation}
where $N$ is a standard normal random variable and $X$ has the density
\begin{equation}\label{eq:symmetric}
    f(x)=\frac{e^3x^2}{2\sqrt{\pi }} e^{-e^2x^2/4} \, \mathds 1\{x>0 \}.
\end{equation}
\end{thm}

In fact, in the following theorem, we give a complete description of the evolution of the process in a $n^{-1/2}$ neighbourhood around the critical point $z_0=1-1/e$. For $0<z<1$ and $n\ge 1$ we let $S(z,n):=S(1,n)\cap [0,z]$ be the set of elements in the list that are smaller than $z$ at time $n$ and let $s(z,n):=|S(z,n)|$.

\begin{thm}\label{theo:11}
As $n$ tends to infinity we have 
\begin{equation*}
  \bigg\{     \frac{s \big( z_0+ yn^{-1/2}, tn \big)}{\sqrt{n}}  \bigg\} _{ \substack{ t>0 \\ y\in \mathbb R  } }  \overset{d}{\longrightarrow } \Big\{ \sqrt{2}e^{-1} B_t+yt -\inf _{x\le t } \big( \sqrt{2}e^{-1} B_x+yx \big) \Big\} _{ \substack{ t>0 \\ y\in \mathbb R  } }.
\end{equation*}
\end{thm}

\begin{remark}\label{rem:1}
The convergence in Theorems~\ref{theo:size}, \ref{theo:12} and  \ref{theo:11} is a weak convergence in the relevant Skorohod spaces. It is equivalent to the following statements. In Theorem~\ref{theo:size}, we define the random function $f_n(t):=n^{-1/2}\big( s(1,\lfloor tn \rfloor ) -tn/e \big)$. Then, for all $M>0$ there exists a coupling of the sequence $f_n$ and a Brownian motion $B_t$ such that almost surely
\begin{equation*}
    \sup _{t\le M} \Big| f_n(t) -\frac{\sqrt{3-e}}{e}\cdot  B_t  \Big| \to 0  ,\quad n\to \infty .
\end{equation*}
In Theorem~\ref{theo:11}, we define $f_n(t,y):=n^{-1/2}s \big( z_0+ yn^{-1/2}, \lfloor tn \rfloor  \big)$. Then, for all $M>0$, there exists a coupling such that almost surely
\begin{equation*}
    \sup _{t,|y|\le M} \Big| f_n(t,y) -\Big( \sqrt{2}e^{-1} B_t+yt -\inf _{x\le t } \big( \sqrt{2}e^{-1} B_x+yx \big) \Big) \Big| \to 0  ,\quad n\to \infty .
\end{equation*}
The convergence in Theorem~\ref{theo:12} is similar.
\end{remark}

\begin{figure}[htp]
    \centering
    \includegraphics[width=7.15cm]{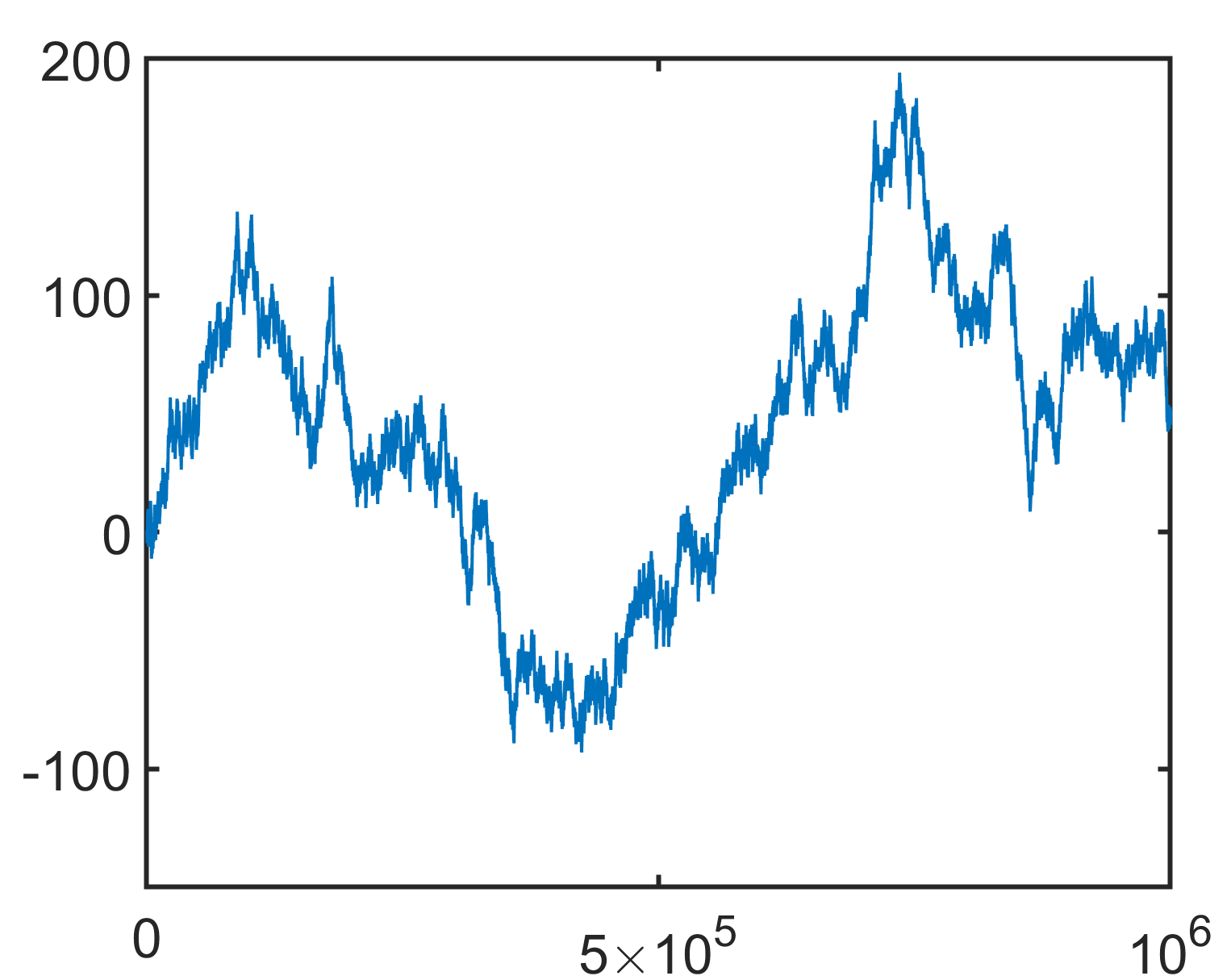}
    \includegraphics[width=7cm]{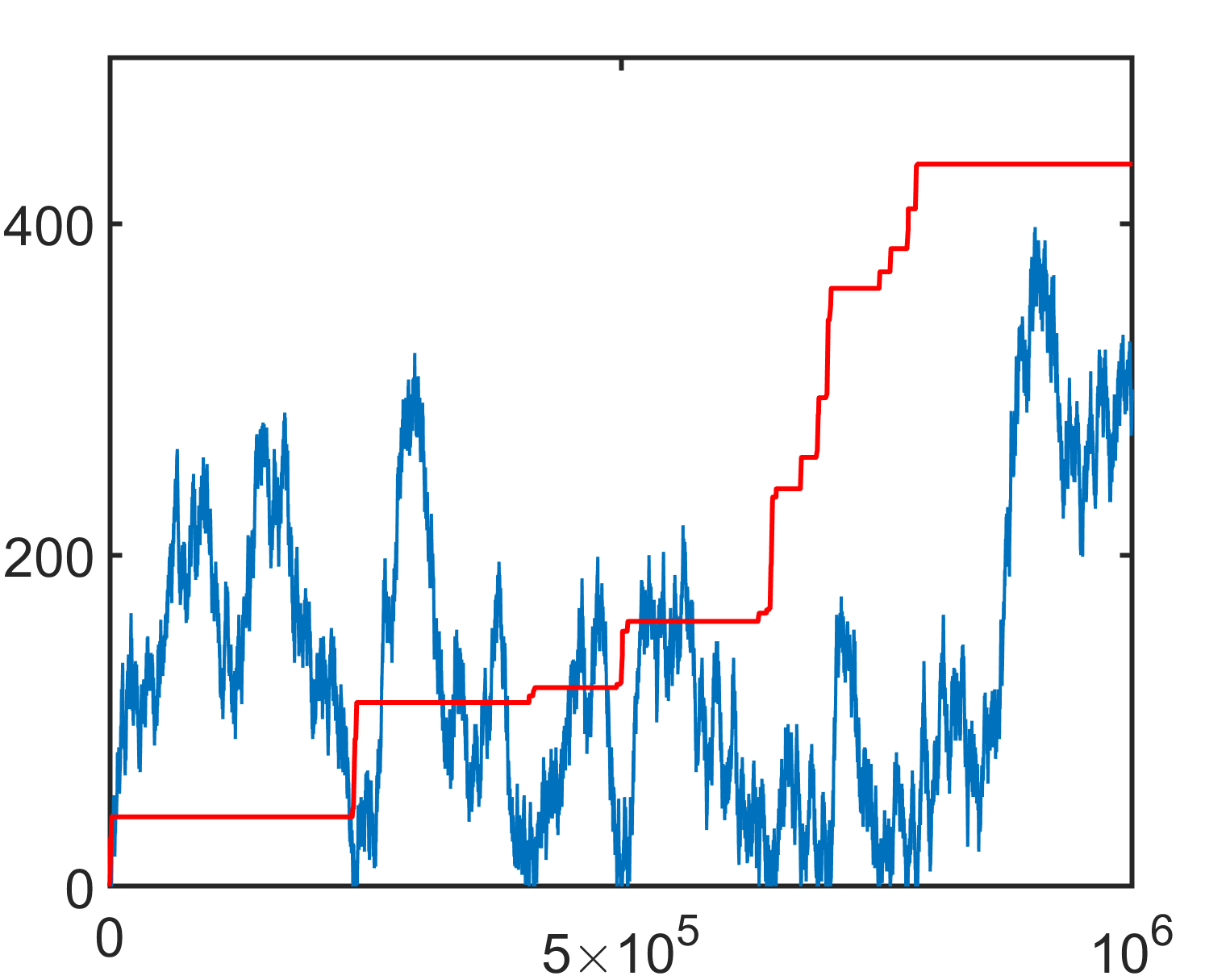}
    \hbox{\hspace{-0.6em}
    \includegraphics[width=14.35cm]{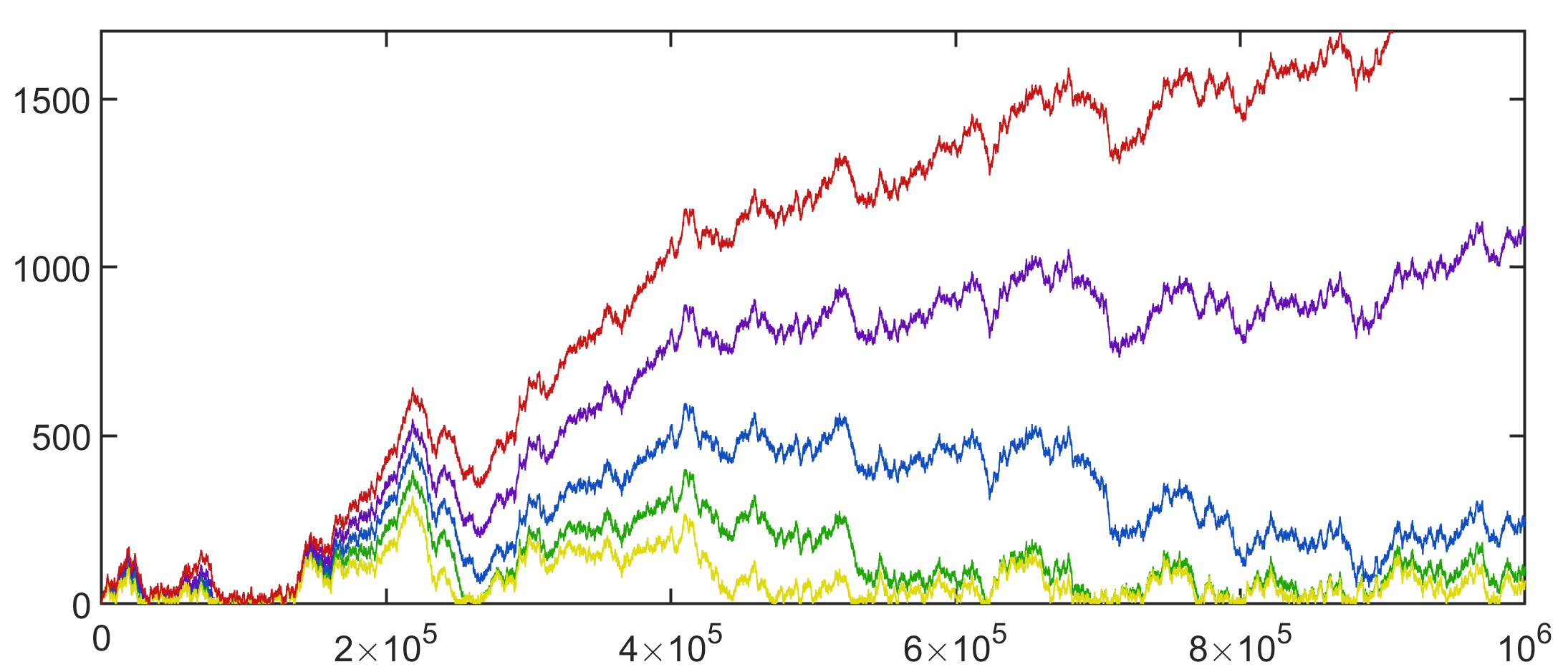}}
    \caption{A computer simulation of the process running for $n=10^6$ steps. Each picture is an independent sample of the process. The first picture shows the centered size of the list $s(1,k)-k/e$ as a function of  $k\le n$. Theorem~\ref{theo:size} states that this is a rescaled Brownian motion. The second picture shows the processes $R_k$ in red and $L_k$ in blue. These processes scale to the maximum process $M_t$ and to $M_t-B_t$ respectively by Theorem~\ref{theo:12}. Note that $R_k$ can grow only when $L_k$ is zero and similarly, in the limit, $M_t$ can grow only when $M_t-B_t$ is zero. The third picture shows $s(z_0+yn^{-1/2},k)$ as a function of $k$ for $y=-2$ in yellow, $y=-1$ in green, $y=0$ in blue, $y=1$ in purple and $y=2$ in red. Theorem~\ref{theo:11} gives the scaling limit of these processes.
    }
    \label{fig:galaxy}
\end{figure}

The rest of the paper is organized as follows. In the next section
we describe the proof of Theorem \ref{t11}. It is based on
subadditivity and martingale concentration applied to the natural
Doob martingale associated with the process. The concentration part
resembles the basic approach in \cite{ADK}. In Section~\ref{sec:martingale} we prove Theorem~\ref{theo:size}, Theorem~\ref{theo:12} and Theorem~\ref{theo:11}. The main idea in the proofs is to consider the process 
\begin{equation}\label{eq:branching explanation}
    W(z,n):=\sum _{x\in S(z,n) }  \frac{1}{1-x} 
\end{equation}
which we call ``the branching martingale." When $S(z,n)$ is not empty, this process has a zero drift when $z=1-1/e$, a negative drift when $z<1-1/e$ and a positive drift when $z>1-1/e$. When $S(z,n)$ is empty, $W$ can only increase in all cases. The intuition behind this magical formula comes from a certain multi-type branching process. In a multi-type branching process, there are individuals of $k$ different types. For each type $i$, an offspring distribution $\mu _{i}$ on $\mathbb N ^k$ is given. In each step, every individual of type $i$ gives birth to a random number of individuals of each type according to the distribution $\mu _{i}$. Recall that a single-type branching process is critical when the expected number of offspring is $1$ and in this case the size of the generation is a martingale. Similarly, for a multi-type branching process one defines the expectation matrix $M$ where $M_{i,j}:=\int x_j d\mu _i(x)$ is the expected number of children of type $j$ of an individual of type $i$. The process is then critical when the maximal eigenvalue of $M$ is one and in this case the process 
\begin{equation}\label{eq:eigen}
    N_t:=u\cdot Z(t)
\end{equation}
is a martingale where $u$ is the eigenvector of $M$ with eigenvalue $1$ and $Z(t)=(Z_1(t),\dots ,Z_k(t))$ is the number of individuals of each type at time $t$. For more background on multi-type branching processes see chapters V and VI in \cite{branching}.

We think of our process as a multi-type branching process in which the type space is continuous. Given $0<z<1$, we define a branching process with the type space $[0,z]$. In this branching, the offspring distribution of an individual of type $x<z$ is given as follows. With probability $(1-z)$, the individual has $0$ offspring, with probability $(z-x)$ it has one offspring of type that is uniformly distributed in $[x,z]$ and with probability $x$ it has two offspring, one of type $x$ and one of type that is uniformly distributed in $[0,x]$. 

We now explain the connection between this branching and our process. Suppose we start the process with a list $S$ such that the minimum of $S$ is $z$, and we want to study the time until $z$ is removed from the list. We think of each point $x<z$ that is added to the list as an individual of type $x$ in the branching process and ignore points that fall to the right of $z$. We claim that the number of individuals that were born before the extinction of the branching described above is exactly the time until $z$ is removed in the minimum process. Indeed, instead of exposing the branching in the usual way according to the generations, we expose the offspring of an individual when it becomes the current minimum of the list. When the element $x$ becomes the minimum, with probability $1-z$ it is removed and $S(z,n)$ decreases by $1$, with probability $z-x$ we remove $x$ from $S(z,n)$ and replace it with a uniform element in $[x,z]$ and with probability $x$ we keep $x$ and add another element to $S(z,n)$ that is uniformly distributed in $[0,x]$.

The branching described above can be shown to be subcritical when $z<z_0$, critical when $z=z_0$ and supercritical when $z>z_0$, where $z_0=1-1/e$. Now, in the critical case we look for a martingale of the form \eqref{eq:eigen} which translates in the continuous case to the process 
\begin{equation}
    N_n:=\sum _{x\in S(z_0,n)}f(x)
\end{equation}
where $f$ is an eigenfunction of the expectation operator with eigenvalue $1$. In order to find $f$ we let $m_n<z_0$ be the minimum of the list at time $n$ and write
\begin{equation}
    \mathbb E \big[ N_{n+1}-N_n \ | \ \mathcal F _n \big]= \int _0^{z_0} f(x)dx -(1-m_n) f(m_n).
\end{equation}
The last expression is zero only when $f(x)=1/(1-x)$ which leads to \eqref{eq:branching explanation}.

\section{Proof of Theorem \ref{t11}}
For any integer $0 \leq \ell \leq n$ and any $x \in [0,1]$
let $S(x,\ell)$ denote the set of numbers that are in $S$ 
and lie in $[0,x]$ after the first
$\ell$ steps, that is, after $x_1,x_2, \ldots , x_{\ell}$ have been
inserted. Put $s(x,\ell)=|S(x, \ell)|$. Therefore $s(1,n)$ is the size
of $S$ at the end of the process.
Let $m_i$ denote the value $\min(S(1,i-1))$ just before $x_i$ enters it.
Thus $m_1=0$ and $m_2=x_1$.

We start with the following lemma showing that
the expected value of the size of $S$ at the end is linear in $n$.
\begin{lemma}
\label{l22}
For all $n\ge 1$ we have 
\begin{equation*}
    \mathbb E [S(1,n)] \ge  n/4
\end{equation*}
\end{lemma}
\vspace{0.2cm}

\noindent
{\bf Proof:}\, For each $1 \leq i \leq n$ let $X_i$ be the indicator 
random variable whose value is $1$ if $x_i <m_i$ (and $0$ otherwise).
Let $Y_i $ be the indicator random variable with value $1$ if
$x_i >m_i$ and $m_i \leq 1/2$. Note that $\sum_{i=1}^n Y_i$ is
at most the number of $x_i$, $0 \leq i <n$ whose values are  in
$[0,1/2]$. Indeed, whenever $Y_i=1$ an element $x_j \leq 1/2$ with $j<i$ leaves the memory set. 
The expected value of this number is 
$1+(n-1)/2$, as $x_0=0<1/2$ and any other $x_i$ lies in
$[0,1/2]$ with probability $1/2$. Therefore the expectation of
$\sum_{i=1}^n Y_i$ is at most $(n+1)/2$. We claim that 
for every $1 \leq i \leq n$, the expected value of 
$2X_i+Y_i$ is at least $1$. Indeed, if $m_i > 1/2$ the expected
value of $2X_i$ is $2m_i >1$ (and the expectation of $Y_i$ is $0$). 
If $m_i \leq 1/2$, the expected value
of $2X_i$ is $2m_i$, and that of $Y_i$ is  $1-m_i$. By linearity of
expectation in this case the expected value of $2X_i+Y_i$ is
$2m_i+(1-m_i)=1+m_i \geq 1$. This proves the claim.
Using again linearity of expectation we conclude that
the expectation of $\sum_{i=1}^n (2X_i+Y_i)$ is at least $n$. Since
the expectation of $\sum_{i=1}^n Y_i$ is at most $(n+1)/2$ it follows
that the expectation of $\sum_{i=1}^n X_i$ is at least $(n-1)/4$. 
Note that the size of $S$ at the end is exactly $1+\sum_{i=1}^n X_i$,
as it has size $1$ at the beginning, its size never decreases,
and it increases by $1$ in step $i$ if and only if  $X_i=1$. 
The assertion of the lemma follows. \hfill $\Box$

We also need the following simple deterministic statement.
\begin{lemma}
\label{l23}
Let $T_1,T_2$ be arbitrary finite disjoint subsets of $[0,1]$.
Let $P=P(n)$ be the process above with the sequence
$x_1,x_2, \ldots ,x_n \in [0,1]$ starting with $S=T_1$. Let
$P'$ denote the process with the same sequence $x_i$, 
starting with $S=T_1 \cup T_2$.
For
$x \in [0,1]$ and $0 \leq \ell \leq n$, let $S(x,\ell)$
denote the set of numbers in $S$ that lie in
$[0,x]$ after the first
$\ell$ steps in the process $P$. Similarly, let $S'(x,\ell)$ denote the
set of numbers in $S$ that lie in $[0,x]$ after the first $\ell$ steps
in the process $P'$. Then $S'(x,\ell)$ contains $S(x,\ell)$ 
and moreover 
$$|S'(x,\ell)|-|S(x,\ell)| \leq |T_2|.$$ 
\end{lemma}
\vspace{0.2cm}

\noindent
{\bf Proof:}\, We first describe the proof for $x=1$.
Apply induction on $\ell$. The result
is trivial for $\ell=0$. Assuming it holds for $\ell$ consider step
number $\ell+1$ in the two processes, when $x=x_{\ell+1}$ enters
the memory $S$. If $x$ is smaller than the minimum of both sets
$S(1,\ell)$ and $S'(1,\ell)$ then it joins the memory in both processes
and the desired properties clearly hold. If it is larger than the minimum
in $S(1,\ell)$ then it is also larger than 
the minimum in $S'(1,\ell)$, it joins
the memory in both cases, and both minima leave. It is easy to check
that the desired properties hold in this case too. Finally, if
$x$ is smaller than the minimum in $S(1,\ell)$ but larger than the
minimum in $S'(1,\ell)$ then $x$ joins both memories and the minimum
of $S'(1,\ell)$ leaves. In this case $S'$ still contains $S$,
but the difference between their sizes decreases by $1$. This also
satisfies the desired properties, completing the proof of the induction
step for $x=1$. The statement for general $x$ follows from the statement for $x=1$
\hfill $\Box$

The above lemma has two useful consequences which are stated and proved
in what follows.
\begin{coro}
\label{c24}
For each fixed
$x$ the function  $f(\ell)=\mathbb E [s(x,\ell)]$  
is subadditive.
\end{coro} 
\vspace{0.2cm}

\noindent
{\bf Proof:}\, We have to show that for every $p,q$, 
$f(p+q)\leq f(p)+f(q)$.
Consider the process $P$ with the 
points $x_1,x_2, \ldots x_{p}, y_1,y_2 , \ldots ,y_q$.
The process starts with the first $p$ steps, let $T$
be the content of the memory $S$ after these $p$ steps.
We can now compare the process $P$ with the points
$y_1,y_2, \ldots y_q$ starting with the empty set $S$
(which is identical to starting with $S=\{0\}$), and the process
$P'$ with the same points $y_i$ starting with $S=T$.
By Lemma \ref{l23}, throughout the two processes, the number of 
elements in the process $P'$ that lie in $[0,x]$ is always at most
the number of elements in $P$ that lie in $[0,x]$, plus $|T|$.
Taking expectations in both sides of this inequality implies the
desired subadditivity. 
\hfill $\Box$

\begin{coro}
\label{c25}
For each $x \in [0,1]$ and any integer $\ell$, 
the random variable $s(x,\ell)$ is $2$-Lipschitz, that is, its value
changes by at most $2$ if we change the value of a single
$x_i$. Therefore 
for any $\lambda>0$ the probability that it deviates from its
expectation by more than $\lambda \sqrt{\ell}$ is at most
$2 e^{-\lambda^2/8}.$
\end{coro} 
\vspace{0.2cm}

\noindent
{\bf Proof:}\, 
The effect that changing the value of $x_i$ to $x'_i$ has on the content
of $S$ after step $i$ is a change of $x_i$ to $x'_i$ in $S$, and possibly
a removal of the minimum element of $S$ before step $i$ in one scenario and
not in the other. In any case this means that one process can be converted
to the other by adding an element to the memory and then by removing one
or two elements from it.
By Lemma \ref{l23} for each
fixed choice of all other $x_j$ the first modification 
can only increase the value of $s(x,\ell)$, and can increase it by
at most $1$. The second modification can then only decrease this value,
decreasing it by at most $2$. This implies that the value changes 
by at most $2$,
that is, the function is $2$-Lipschitz. The concentration inequality
thus follows from Azuma's Martingale Inequality, c.f. \cite[Theorem~1.1]{bercu2008exponential}.  \hfill $\Box$
\vspace{0.2cm}

\noindent
{\bf Proof of Theorem \ref{t11}:}\, 
The symmetric difference of the sets $S(1,n)$ and $S(1,n+1)$ is of size at most two and therefore we may assume that $n$ is even throughout the proof of the theorem. We may also assume that $n$ is sufficiently large. Without any attempt  to optimize the absolute constants, let
$y \in [0,1]$ be a real number so that the expected value
$\mathbb E [s(y,n/2)]$ of the number of elements in the memory $S$ 
that lie in $[0,y]$ at the end 
of the process with the first $n/2$ elements $x_i$ is $5 \sqrt {n \log n}$. 
Since $\mathbb E [s(0,n/2)]=0$, $\mathbb E [s(1,n/2)]>n/8$ by Lemma \ref{l22}, 
and the function $f(y)=\mathbb E [s(y,n/2)]$ is continuous, there is such a
value (provided $n/8 \geq 5 \sqrt {n \log n}$.)
We claim that the expected 
number of indices $i$ so that $n/2 <i \leq n$
and the minimum in the memory satisfies
$m_i > y$, is smaller than $1/n^2$. Indeed, it follows from Lemma~\ref{l23} that the function
$\mathbb E [s(y,m)]$ is monotone increasing  in $m$. Thus, for every $\ell \in (n/2,n]$, 
$\mathbb E [s(y,\ell -1)] \geq 5 \sqrt {n \log n}.$
By Corollary \ref{c25} (and using that $\{m_\ell >y\}=\{s(y,\ell -1)=0\}$) this implies that
for every fixed $\ell$ in this range the probability that $m_{\ell}$
is larger than $y$ is at most
$2e^{-25 n \log n/(8 \ell)} <1/n^{3}$. 
Therefore, by linearity of expectation,
the expected number of steps $i \in (n/2,n]$
in which $m_i>y$ is smaller than $1/n^2$, proving the claim.

By this claim, the expected number of steps $i \in
(n/2,n]$ so that $m_i>y$ and $x_i >m_i$
is, of course, also at most $1/n^2$.

On the other hand, by the subadditivity established in Corollary
\ref{c24}, the expected value of $s(y,n)$ is at most 
$10 \sqrt{n \log n}$. 

We next show that $y$ deviates from $1-1/e$ by 
$O(\frac{\sqrt {\log n}}{\sqrt n}).$ To this end we compute the
expectation of the random variable $Y$ counting the number of
indices $i \in (n/2,n]$ for which $x_i \leq y$ and 
$x_i$ is being removed at a step in which the arriving element
exceeds $y$. By definition $Y=\sum_{n/2<i \leq n} X_i$ where
$X_i$ denotes the indicator random variable whose value is
$1$ if $x_i \leq y$ and $x_i$ is being removed 
at a step in which the arriving element is larger than $y$.

For all $n/2<i\le n$ and $i<j \le n$, let $\mathcal  A _{i,j}$ be the event that $x_i$ is removed by $x_j$. Conditioning on $x_i$ and on the event $\mathcal A _{i,j}$, the random variable $x_j$ is uniformly distributed on the interval $[x_i,1]$ and therefore on the event $\{x_i\le y\}$ we have
\begin{equation}
    \mathbb P \big( x_j\ge y \mid x_i , \mathcal A _{i,j} \big)= \frac{1-y}{1-x_i}.
\end{equation}
Thus, we obtain 
\begin{equation*}
\begin{split}
    \mathbb E [ X_i ] &=\mathbb E \bigg[ 
\mathds 1 \{x_i\le y \} \mathbb P \bigg( \bigcup _{j=i+1}^n \mathcal A _{i,j} \cap \{ x_j\ge y \}  \mid x_i \bigg) \bigg] \\
&=\mathbb E \bigg[ 
\mathds 1 \{x_i\le y \} \sum _{j=i+1}^n \mathbb P \big( \mathcal A _{i,j} \mid x_i \big) \mathbb P \big( x_j\ge y \mid x_i , \mathcal A _{i,j}  \big) \bigg] \\
&=\mathbb E \Big[ 
\mathds 1 \{x_i\le y \} \frac{1-y}{1-x_i}  \cdot \mathbb P \big( x_i \text{ is removed before time $n$} \mid x_i \big)  \Big]
\end{split}
\end{equation*}
The expected number of 
$i \in (n/2,n]$ for which $x_i\le y$ is not removed is 
at most $10 \sqrt {n \log n}$, as all these elements belong
to $S(y,n)$. Thus,
\begin{equation*}
    \begin{split}
        \mathbb E [Y]&= \sum_{n/2 <i \leq n} \mathbb E[X_i]= O(\sqrt{n\log n})
+\sum_{n/2 <i \leq n} \mathbb E \Big[ 
\mathds 1 \{x_i\le y \} \frac{1-y}{1-x_i}  \Big]\\
&=O(\sqrt{n\log n }) +\frac{n}{2}\int_{0}^{y} \frac{1-y}{1-x} dx=O(\sqrt{n\log n }) +\frac{n}{2} (1-y) \log \Big( \frac{1}{1-y}\Big),
    \end{split}
\end{equation*}
where in the third equality we used that $x_i$ is uniform in $[0,1]$. 

On the other hand, 
the expected number of steps $j \in (n/2,n]$
in which the element $x_j>y$ arrived and caused the removal of an
element $x_i=m_j<y$ for some $i \in (n/2,n]$ is
$(n/2)(1-y)-O(\sqrt{n \log n})$. Indeed, the expected number of $j \in
(n/2,n]$ with $x_j >y$ is $(n/2)(1-y)$, and almost 
each such $x_j$ removes an
element $m_j$, where the expected number of such $m_j$ that exceed
$y$ is $O(1/n^2)$. In some of these steps the removed element may be
$x_i$ for some $i \leq n/2$, but the expected number of such
indices $i$ is at most the expectation of $s(y,n/2)$ which is
only $O(\sqrt {n \log n})$.  

It follows that
$$
\frac{n}{2} (1-y) \log \Big( \frac{1}{1-y}\Big)
=\frac{n}{2}(1-y)+O(\sqrt {n \log n}).
$$
Dividing by $n(1-y)/2$ (noting that 
$1-y$ is bounded away from $0$) we conclude that 
$y=1-1/e+O(\frac{\sqrt {\log n}}{\sqrt n}).$

Since the expected number of elements
$x_i$ for $i \in (n/2,n]$ 
that fall in the interval 
$$\Big[ y-O\Big( \frac{\sqrt {\log n}}{\sqrt n} \Big), y+O\Big( \frac{\sqrt {\log
n}}{\sqrt n} \Big) \Big] $$
is $O(\sqrt{n \log n})$ we conclude that the expected 
number of steps $i \in (n/2,n]$ satisfying $x_i >1-1/e$
that leave the memory during the process is 
$O(\sqrt {n \log n})$. In addition, since
$\mathbb E [s(y,n)] \leq 10 \sqrt{n \log n}$ it follows
that the expected number of steps $i \in (n/2,n]$ so that $x_i
<1-1/e$  and $x_i$ stays in the memory at the end of the process
is also at most $O(\sqrt {n \log n})$.

By splitting the set of all $n$ steps into dyadic intervals we
conclude that the expected size of the symmetric difference 
between the final content of the memory and the set of all
elements $x_i$ larger than $z=1-1/e$ is $O(\sqrt {n \log n})$. 
This clearly also implies that the expected value of $s(1,n)$
deviates from $n/e$ by $O(\sqrt {n \log n})$. 
Finally note that by Corollary \ref{c25},
for any positive $\lambda$ 
the probability that either $s(1,n)$ or the above mentioned
symmetric difference deviate from their expectations
by more than $\lambda \sqrt n$ is at most
$O(e^{-\Omega(\lambda^2)})$.
\hfill $\Box$
\vspace{0.2cm}

\section{The branching martingale}\label{sec:martingale}

In this section we prove Theorem~\ref{theo:size}, Theorem~\ref{theo:12} and Theorem~\ref{theo:11}. The main idea in the proof is the following martingale which we call the branching martingale.

Let $0<z<1$ and let $z_0:=1-1/e$ be the critical point. Recall that $S(z,n)$ is the set of elements in the list at time $n$ that are smaller than $z$. Define the processes 
\begin{equation*}
    W(z,n):=\sum _{x\in S(z,n) }  \frac{1}{1-x} \quad \quad  Z(z,n):=\sum _{k=1}^n W(z,k) \cdot \mathds 1 \big\{ W(z,k-1)=0 \big\} 
\end{equation*}
 and let $X(z,n):=W(z,n)-Z(z,n)$. The following claim is the fundamental reason to consider these processes.
 \begin{prop}\label{prop:22}
 On the event $\{W(z,n)\neq 0\}$ we have that
 \begin{equation*}
     \mathbb E \big[ W(z,n+1)-W(z,n) \, | \, \mathcal F _n  \big]=-\log (1-z)-1
 \end{equation*}
 and therefore $X(z_0,n)$ is a martingale. Moreover, $Z$ is roughly minus the minimum process of $X$. More precisely, we have almost surely
 \begin{equation}\label{eq:12}
     0\le Z(z,n) +\min _{k\le n} X(z,k) \le \frac{1}{1-z}.
\end{equation}
\end{prop}
 
\begin{remark}
Proposition~\ref{prop:22} essentially explains all the results in this paper. See Theorem~\ref{theo:11} for example. Since $X(z_0,n)$ is a martingale it is a Brownian motion in the limit and $Z(z_0,n)$ is minus the minimum process of this Brownian motion. Thus, $W(z_0,n)$, which is closely related to the number of elements in the list that are smaller than $z_0$ is the difference between the Brownian motion and its minimum.  
\end{remark}
 
\begin{proof}[Proof of Proposition~\ref{prop:22}]
On the event $\{ W(z,n)\neq 0 \}$ we have   
\begin{equation}\label{eq:W}
    W(z,n+1)-W(z,n) = \mathds 1 \{ x_{n+1} \le z \} \frac{1}{1-x_{n+1}} -\mathds 1 \{  x_{n+1} \ge m_n \} \frac{1}{1-m_n},
\end{equation}
where $m_n$ is the minimum of the list at time $n$ and $x_{n+1}$ is the uniform variable that arrives at time $n+1$ in the process. Thus, on the event $\{ W(z,n)\neq 0 \}$ we have 
\begin{equation*}
    \mathbb E \big[ W(z,n+1)-W(z,n) \ \big| \ \mathcal F _n \big]= \int _0^{z} \frac{1}{1-x} dx - \int _{m_n}^1 \frac{1}{1-m_m} =-\log (1-z)-1.
\end{equation*}
It follows that $X(z_0,n)$ is a martingale. Indeed on the event $\{ W(z,n)\neq 0 \}$ we have
\begin{equation*}
    \mathbb E \big[ X(z,n+1)-X(z,n) \ \big| \ \mathcal F _n \big]=\mathbb E \big[ W(z,n+1)-W(z,n) \ \big| \ \mathcal F _n \big]=-\log (1-z_0)-1=0.
\end{equation*}
Whereas on the event $\{ W(z,n)= 0 \}$ we have that $X(n+1,z)=X(n,z)$.

We next turn to prove the second statement of the proposition. For all $k\le n $ we have that
\begin{equation*}
X(z,k)=W(z,k)-Z(z,k)\ge -Z(z,k) \ge -Z(z,n),   
\end{equation*}
and therefore $\min _{k\le n} X(z,k) \ge -Z(z,n)$.
Next, if $W(z,n)=0$ then 
\begin{equation*}
    \min _{k\le n} X(z,k) \le X(z,n) = -Z(z,n).
\end{equation*}
If $W(z,n) \neq 0$, we let $n'\le n$ be the last time before $n$ for which $W(z,n')=0$ (assuming it exists). We have that 
\begin{equation*}
 \min _{k\le n} X(z,k) \le  X(z,n') = -Z(z,n') \le -Z(z,n) +\frac{1}{1-x_{n'+1}} \le -Z(z,n) +\frac{1}{1-z}.
\end{equation*}
Finally suppose there is no such $n'\le  n$ for which $W(z, n') = 0$. Then $Z(z, k) = 0$ for all $k\le n$, and hence $X(z, k) = W(z, k)$ for all $k$, and $Z(z, n) + \min _{k\le n} X(z, k) = 1\le 1/(1-z)$. This finishes the proof of the proposition. 
\end{proof}

In the following section we give some rough bounds on the process that hold with very high probability. We will later use these a priori bounds in order to obtain more precise estimates and in order to prove the main theorems. 
 
\subsection{First control of the process}
 
\begin{definition}
 We say that an event $\mathcal A$ holds with very high probability (WVHP) if there are absolute constants $C,c>0$ such that $\mathbb P(\mathcal A )\ge 1-C\exp(-n^c)$. 
\end{definition} 

Let $\epsilon :=0.01$ and recall that $z_0:=1-1/e$. Define the events 
\begin{equation*}
    \mathcal A _1 := \big\{ \forall z < z_0, \ s(z,n)\le n^{\epsilon }/(z_0-z) \big\}
\end{equation*}
and
\begin{equation*}
    \mathcal A _2:= \big\{ \forall z \ge z_0, \ s(z,n)\le (z-z_0)n+ n^{1/2+\epsilon }   \big\}.
\end{equation*}

\begin{lemma}\label{lem:A_1 A_2}
The event $\mathcal A _1 \cap \mathcal A _2$ holds WVHP.
\end{lemma}

\begin{proof}
Let $z< z_0$ and note that $W(z,n)\ge s(z,n)$. Also note that if $W(z,k-1)=0$ and $W(z, k)\neq 0$, then $W(z,k)=1/(1-x)$ for some
$x\le z<z_0$. Thus, we have that 
\begin{equation}\label{eq:14}
    \big\{ s(z,n) \ge n^{\epsilon }/(z_0-z) \big\} \subseteq \big\{ W(z,n)\ge n^{\epsilon }/(z_0-z)  \big\} \subseteq \bigcup _{ k\le n} \mathcal C _{k} 
\end{equation}
where the events $\mathcal C _k$ are defined by 
\begin{equation*}
\mathcal C_k = \big\{  0<W(z, k) \le e ,\ \forall k\le m \le n, W(z,m)\neq 0 \text{ and } W(z,n)\ge n^{\epsilon }/(z_0-z)   \big\}. 
\end{equation*}
Next, we show that each of these events has a negligible probability. To this end, let $k\le n$ and define
\begin{equation*}
    M_m:=W(z,m)+(m-k)(\log (1-z)+1),\quad  \tau :=\min \{m\ge k : W(z,m)=0\}.
\end{equation*}
By Proposition~\ref{prop:22}, the process $M_{m\wedge \tau }$ is a martingale. Moreover, since $\log (1-z)+1\ge z_0-z$, we have that 
\begin{equation*}
\begin{split}
    \mathcal C _k &\subseteq  \big\{  M_{n\wedge \tau }-M_{k\wedge \tau }\ge n^{\epsilon }/(z_0-z) -e+(n-k) (\log (1-z)+1) \big\} \\
    &\subseteq \big\{  M_{n\wedge \tau }-M_{k\wedge \tau }\ge  n^{\epsilon/2 }/(z_0-z)+(n-k) (z_0-z) \big\}.
\end{split}
\end{equation*}
Thus, using that $M_{m\wedge \tau }$ has bounded increments and that $(a+b)^2\ge a^2+b^2$, we get from Azuma's inequality that  
\begin{equation*}
    \mathbb P \big( \mathcal C _k \big) \le \exp \Big( -c\ \frac{n^{\epsilon }(z_0-z)^{-2}+(n-k)^2(z_0-z)^2}{n-k} \Big) \le C\exp ( -cn^{c} ), 
\end{equation*}
where the last inequality clearly holds both when $n-k\le n^{\epsilon /2} (z_0-z)^{-2}$ and when $n-k\ge n^{\epsilon /2} (z_0-z)^{-2}$. It follows from \eqref{eq:14} that
\begin{equation*}
    \mathbb P \big( s(z,n) \ge n^{\epsilon }/(z_0-z)  \big) \le C\exp (-n^c). 
\end{equation*}
In order to obtain this inequality simultaneously for all $z\le z_0$ we first use a union bound to get it simultaneously for all $z\in \{i/n: i\in \mathbb N , i/n\le z_0 \}$ and then argue that $s(z,n)$ does not change by much WVHP when $z$ changes by $O(n^{-1})$. The details are omitted. It follows that $\mathbb P (\mathcal A _1) \ge 1 -C\exp (-n^c)$. 

We turn to bound the probability of $\mathcal A _2$. To this end, let $z\ge z_0$ and $z_1:=z_0-n^{-1/2}$. The event 
\begin{equation*}
    \big\{ s(z_1,n) \le  n^{1/2+\epsilon }/2   \big\} \cap  \big\{ \big| \big\{ k\le n : x_k\in [z_1,z] \big\} \big| \le (z-z_0)n+ n^{1/2+\epsilon  }/2   \big\} 
\end{equation*}
holds WVHP by the first part of the proof and by Chernoff's inequality for the Binomial random variable $\big| \big\{ k\le n : x_k\in [z_1,z] \big\} \big| \sim \text{Bin}(n,z-z_1)$. On this event we have that $s(z,n) \le (z-z_0)n+n^{1/2+\epsilon }$ and therefore, using the discretization argument as above we obtain $\mathbb P (\mathcal A _2)\ge 1-C\exp (-n^c)$.
\end{proof}

The following lemma implies in particular that, with very high probability, the minimum is rarely above $z$ for any $z>z_0$. For the statement of the lemma, define the random variable 
\begin{equation*}
    M_n(z):=\big| \big\{ k\le n : m_k\ge z \big\} \big|
\end{equation*}
and the events
\begin{equation*}
    \mathcal B _1 := \big\{ \forall z > z_0, \ M_n(z)\le n^{\epsilon }/(z-z_0) \big\}
\end{equation*}
and 
\begin{equation*}
    \mathcal B _2 := \big\{ \forall z \le z_0, \ \big|M_n(z)- (\log (1-z)+1)n\big| \le n^{1/2+\epsilon } \big\}.
\end{equation*}

\begin{lemma}\label{lem:B_1 B_2}
The event $\mathcal B _1 \cap \mathcal B _2$ holds WVHP.
\end{lemma}
\begin{proof}
    We claim that for any $z<1$ the process 
    \begin{equation}
        L(z,k):=W(z,k)+(\log (1-z)+1)(k-1)-M_{k-1}(z)
    \end{equation}
    is a martingale. Indeed, if $W(z,k)\neq 0$ then $M_{k}(z)=M_{k-1}(z)$ and therefore by Proposition~\ref{prop:22} we have $\mathbb E [L(z,k+1)-L(k,z)\mid \mathcal F _k]=0$. Otherwise, using that $M_k(z)=M_{k-1}(z)+1$ we have 
    \begin{equation*}
        \mathbb E [L(z,k+1)-L(k,z)\mid \mathcal F _k]=\int _0^z \frac{1}{1-x}dx +\log (1-z)=0.
    \end{equation*}
    Moreover, when $z<3/4$ the martingale $L(z,k)$ has bounded increments and therefore by Azuma's inequality we have WVHP 
    \begin{equation}\label{eq:64544}
       \big|  W(z,n+1)+(\log (1-z)+1)n-M_{n}(z) \big| \le n^{1/2+\epsilon }.
    \end{equation}
    Now, when $z\le z_0$, on the event $\mathcal A _2$ we have 
    \begin{equation*}
        W(z,n+1) \le W(z_0,n+1) \le   e \cdot s(z_0,n+1) \le n^{1/2+\epsilon }.
    \end{equation*}
    Substituting this back into \eqref{eq:64544} we obtain that WVHP
\begin{equation}\label{eq:M}
    \big| M_n(z)-(\log (1-z)+1)n \big|\le n^{1/2+2\epsilon }.
\end{equation}
This shows that $\mathcal B_2$ holds WVHP using the discretization trick. Indeed, WVHP each $z \in [0,z_0]$ is the minimum at most $n^{\epsilon }$ times and therefore WVHP $M_n(z)$ won't change by more than $n^{2\epsilon }$ when $z$ changes by $O(n^{-1})$.  

We turn to bound the probability of $\mathcal B _1$. To this end, let $n_0:= \lfloor n^{\epsilon }(z-z_0)^{-2} \rfloor $ and let $z_0 \le z \le 3/4$. Using \eqref{eq:64544} and the fact that $\log (1-z) +1\le -e(z-z_0)$ we obtain that WVHP for all $n_0\le k\le n $ 
\begin{equation}
    W(z,k+1) \ge e(z-z_0)k -|L(z,k)| \ge e(z-z_0)k -k^{1/2+\epsilon }>0,
\end{equation}
where the last inequality follows from the definition of $n_0$.
Thus, WVHP, for all $n_0< k \le n$ we have that $m_k \le z$ and therefore 
\begin{equation*}
    M_n(z)=M_{n_0}(z) \le M_{n_0}(z_0) \le n_0^{1/2+\epsilon } \le n^{\epsilon } /(z-z_0),
\end{equation*}
where the second to last inequality holds WVHP by \eqref{eq:M} with $z=z_0$. Of course the same bound holds when $3/4\le z\le 1$ (by slightly changing $\epsilon $) and therefore, using the discretization trick we get that $\mathcal B _1$ holds WVHP.
\end{proof}

\subsection{Proof of Theorem~\ref{theo:11}}

The main step toward proving the theorem is the following proposition that determines the scaling limit of $X(z,n)$ when $z$ is in a small neighborhood around $z_0=1-1/e$. 
 
 \begin{prop}\label{prop:1}
 We have that 
\begin{equation*}
     \Big\{ \frac{X(z_0+yn^{-1/2}\, ,\, tn)}{\sqrt{n}}  \Big\} _{ \substack{ t>0 \\ y\in \mathbb R  } } \overset{d}{\longrightarrow} \big\{  \sqrt{2} B_t+eyt \big\}_{ \substack{ t>0 \\ y\in \mathbb R  } } , \quad n \to \infty .
\end{equation*}  
\end{prop}

We start by proving Theorem~\ref{theo:11} using Proposition~\ref{prop:1}. 

\begin{proof}[Proof of Theorem~\ref{theo:11}]
By Proposition~\ref{prop:1} and \eqref{eq:12} we have that the joint distribution
\begin{equation*}
     \Big\{ \Big( \frac{X(z_0+yn^{-1/2}\, ,\, tn)}{\sqrt{n}} \ , \ \frac{Z(z_0+yn^{-1/2}\, ,\, tn)}{\sqrt{n}}  \Big) \Big\}_{ \substack{ t>0 \\ y\in \mathbb R  } }
\end{equation*}
converges to $\big(  \sqrt{2} B_t+eyt \, , \,  -\inf _{s\le t }(\sqrt{2} B_t+eys) \big)$. Thus, using that $W(z,n)=X(z,n)+Z(z,n)$ we obtain
\begin{equation}\label{eq:13}
     \Big\{  \frac{W(z_0+yn^{-1/2}\, ,\, tn)}{\sqrt{n}} \Big\} _{ \substack{ t>0 \\ y\in \mathbb R  } } \overset{d}{\longrightarrow} \big\{  \sqrt{2} B_t+eyt - \inf _{s\le t } \big( \sqrt{2} B_t+eys \big)  \big\}_{ \substack{ t>0 \\ y\in \mathbb R  } }.
\end{equation}
For the proof of Theorem~\ref{theo:11} it suffices to show that $W(z_0+yn^{-1/2}\, ,\, k)$ in the equation* above can be replaced by $e\cdot s(z_0+yn^{-1/2},k)$. Intuitively, it follows as most of the terms in the sum in the definition of $W$ are close to $e$. Formally, let $z_1:= z_0-n^{-1/4}$ and  $z_2:=z_0+n^{-1/2+\epsilon }$. By Lemma~\ref{lem:A_1 A_2} we have WVHP that for all $z_1\le z \le z_2$ and $k\le n^{1+\epsilon }$ 
\begin{equation}\label{eq:Ws}
\begin{split}
     \big| W(z, k)-e\cdot s(z,k)  \big| &\le C s(z_1,k) +\sum _{x\in S(z,k)\setminus S(z_2,k)} \Big| \frac{1}{1-x}-e \Big|   \\
     &\le C n ^{1/4+\epsilon } +C n^{-1/4 } s(z_2,k) \le n^{1/4+2\epsilon }.
\end{split}
\end{equation}
This shows that $W$ can be replaced by $e\cdot s$ in equation \eqref{eq:13} and therefore finishes the proof of the theorem.
\end{proof}

We turn to prove Proposition~\ref{prop:1}. The proposition follows immediately from the following two propositions.

\begin{prop}\label{prop:2}
 We have that 
\begin{equation*}
     \Big\{ \frac{X(z_0,tn)}{\sqrt{n}}  \Big\} _{t>0} \overset{d}{\longrightarrow} \big\{  \sqrt{2} B_t \big\}_{t>0}, \quad n \to \infty .
\end{equation*} 
\end{prop}

\begin{prop}\label{prop:3}
We have WVHP for all $z$ with $|z-z_0|\le n^{-1/2+\epsilon }$ that
 \begin{equation*}
      \big| X(z,n) -X(z_0,n) -en(z-z_0) \big|\le n^{1/4+3\epsilon }    .
 \end{equation*}
\end{prop}

We start by proving Proposition~\ref{prop:2}. The main tool we use is the following martingale functional central limit theorem \cite[Theorem~8.2.8]{durrett2019probability} . 

\begin{theo}\label{theo:durrett}
Let $M_n$ be a martingale with bounded increments. Recall that the predictable quadratic variation of  $M_n$ is given by  
\begin{equation*}
    V_n:=\sum _{k=1}^{n} \mathbb E \big[ \big( M_k-M_{k-1} \big)^2 \ | \ \mathcal F _{k-1} \big].
\end{equation*}
Suppose that for all $t>0$ we have that
\begin{equation*}
    \frac{V_{tn}}{n}  \overset{p}{\longrightarrow} \sigma ^2t, \quad n\to \infty .
\end{equation*}
where the convergence is in probability. Then
\begin{equation*}
     \Big\{ \frac{M_{tn}}{\sqrt{n}}  \Big\} _{0<t<1} \overset{d}{\longrightarrow} \big\{  \sigma B_t \big\}_{0<t<1} ,
\end{equation*}
where $B_t$ is a standard Brownian motion.
\end{theo}

In order to use Theorem~\ref{theo:durrett} we need to estimate the predictable quadratic variation of the martingale $X(z_0,n)$. To this end, we need the following lemma. Let us note that the exponent $3/4$ in the lemma is not tight. However, it suffices for the proof of the main theorems.

\begin{lemma}\label{cor:1}
For any bounded function $f:[0,1]\to \mathbb R $ that is differentiable on $[0,z_0]$ we have
\begin{equation*}
   \mathbb P \Big(  \Big| \sum _{k=1}^n f(m_k) - n\int _0^{z_0} \frac{f(x)}{1-x}dx \Big| \le  n^{3/4+2\epsilon } \Big) \ge 1-C\exp ( -n^c ),
\end{equation*}
where the constants $C,c$ may depend on the function $f$.
\end{lemma}

\begin{proof}
Throughout the proof we let the constants $C$ and $c$ depend on the function $f$. For the proof of the lemma we split the interval $[0,z_0]$ into $\lfloor n^{1/4} \rfloor $ small intervals $[y_{i-1},y_i]$ where $y_i:=iz_0/\lfloor n^{1/4} \rfloor $. For all $i\le \lfloor n^{1/4}\rfloor $ we let 
\begin{equation}
    J_i:= \big\{ k\le n : m_j\in [y_{i-1},y_i] \big\}.
\end{equation}
We have that 
\begin{equation}\label{eq:385}
\begin{split}
    \Big| \sum _{k=1}^n f(m_k) &-n\int _0^{z_0} \frac{f(x)}{1-x}dx   \Big|  \\
    &\le \sum _{i=1}^{\lfloor n^{1/4} \rfloor } \Big| \sum _{k\in J_i} f(m_k)- n\int _{y_{i-1}}^{y_i} \! \frac{f(x)}{1-x}dx  \Big| + \sum _{k=1}^n f(m_k) \cdot \mathds 1 \{ m_k \ge z_0\}.
\end{split}
\end{equation}
Next, we estimate the size of $J_i$. Using Lemma~\ref{lem:B_1 B_2} we have WVHP
\begin{equation*}
\begin{split}
    |J_i|=M_n(y_{i-1})-M_n(y_i)&=n\big( \log (1-y_{i-1})-\log (1-y_i) \big) +O(n^{1/2+\epsilon })\\
    &= \frac{n(y_i-y_{i-1})}{1-y_i}  +O(n^{1/2+\epsilon }),
\end{split}
\end{equation*}
where in the last equality we used the Taylor expansion of the function $\log (1-x)$ around $x=y_i$. Thus, using that $f(x)=f(y_i) +O(n^{-1/4})$ and $f(x)/(1-x)=f(y_i)/(1-y_i) +O(n^{-1/4})$ for all $x\in [y_{i-1},y_i]$ we obtain 
\begin{equation*}
\begin{split}
    \Big| \sum _{k\in J_i} f(m_k)- n\int _{y_{i-1}}^{y_i} \! \frac{f(x)}{1-x}dx  \Big| &\le |f(y_i)| \cdot \Big|  |J_i|- \frac{n(y_i-y_{i-1})}{1-y_i}  \Big| +O\big( |J_i|n^{-1/4} +n^{1/2} \big)\\
    &=O(n^{1/2+\epsilon }).
\end{split} 
\end{equation*}
Substituting the last estimate back into \eqref{eq:385} we obtain 
\begin{equation*}
    \Big| \sum _{k=1}^n f(m_k) -n\int _0^{z_0} \frac{f(x)}{1-x}dx   \Big|  \le O(n^{3/4+\epsilon }) +C M_n(z_0) =O(n^{3/4+\epsilon }),
\end{equation*}
where in the last equality we used Lemma~\ref{lem:B_1 B_2} once again. 
\end{proof}

We can now prove Proposition~\ref{prop:2}.
 
\begin{proof}[Proof of Proposition~\ref{prop:2}]
In order to use Theorem~\ref{theo:durrett} we compute the predictable quadratic variation of $X(z_0,k)$. Clearly, on the event $\{m_k\ge z_0\}$ we have that $\mathbb E \big[ ( X(z_0,k+1)-X(z_0,k) )^2  \ \big| \ \mathcal F_k \big]=0$. On the event $\{m_k\le z_0\}$ we have 
\begin{equation*}
\begin{split}
    \mathbb E \big[ ( X&(z_0,k+1)-X(z_0,k) )^2  \ \big| \ \mathcal F_k \big]=\mathbb E \big[ ( W(z_0,k+1)-W(z_0,k) )^2  \ \big| \ \mathcal F_k \big]  \\
    &= \int _0^{z_0} \frac{1}{(1-x)^2} dx-\frac{2}{1-m_k} \int _{m_t}^{z_0} \frac{1}{1-x}dx 
    +\frac{1}{1-m_k} \\
    &= \frac{1}{1-x} \Big| _0^{z_0}+\frac{2}{1-m_k} \log (1-x) \Big| _{m_k}^{z_0} 
    +\frac{1}{1-m_k}=e-1 -\frac{ 2\log (1-m_k) +1}{1-m_k}.
\end{split}
\end{equation*}
Thus, the predictable quadratic variation of $X(z_0,k)$ is given by 
\begin{equation}\label{eq:quadratic}
    V_n:=\sum _{k=0}^{n-1} \mathbb E \big[ ( X(z_0,k)-X(z_0,k) )^2  \ \big| \ \mathcal F_k \big]=\sum _{k=1}^{n-1} f(m_k),
\end{equation}
where 
\begin{equation*}
f(x):= \Big(e-1 -\frac{ 2\log (1-x) +1}{1-x} \Big) \cdot  \mathds 1 \{ x\le z_0\} .   
\end{equation*}
Next, we have that 
\begin{equation*}
\begin{split}
    \int _0^{z_0} \frac{f(x)}{1-x} &dx = \int _0^{z_0} \frac{1}{1-x}\Big(e-1 -\frac{ 2\log (1-x) +1}{1-x} \Big)dx \\
    &=(1-e)\log (1-x) \Big| _0^{z_0} -\frac{2\log (1-x)+3}{1-x} \Big| _0^{z_0}= e-1 -e+3=2
\end{split}
\end{equation*}
and therefore, by Lemma~\ref{cor:1} we have WVHP that $|V_n-2n|\le Cn^{3/4+3\epsilon }$. It follows that for all $t>0$ we have $V_{tn}/n \to 2t$ in probability and therefore by Theorem~\ref{theo:durrett} we have 
\begin{equation*}
     \Big\{ \frac{X(z_0,tn)}{\sqrt{n}}  \Big\} _{t>0} \overset{d}{\longrightarrow} \big\{  \sqrt{2} B_t \big\}_{t>0}
\end{equation*}
as needed.
\end{proof}

We turn to prove Proposition~\ref{prop:3}. The main tool we use is the following martingale concentration result due to Freedman \cite{freedman1975tail}. See also \cite[Theorem~1.2]{bercu2008exponential}.

\begin{theo}[Freedman's inequality]\label{theo:freedman}
    Let $M_n$ be a martingale with increments bounded by $M$ and let
    \begin{equation*}
        V_n:=\sum _{k=1}^n \mathbb E \big[ (M_{k+1}-M_k)^2 \ | \ \mathcal F _k \big]
    \end{equation*}
    be the predictable quadratic variation. Then, 
    \begin{equation*}
        \mathbb P \big( |M_n|\ge x, \ V_n\le y \big) \le 2\exp \Big( -\frac{x^2}{2y+2Mx} \Big).
    \end{equation*}
\end{theo}

We can now prove Proposition~\ref{prop:3}.

\begin{proof}[Proof of Proposition~\ref{prop:3}]
Let $z$ such that $|z-z_0|\le n^{-1/2+\epsilon }$. It follows from Proposition~\ref{prop:22} that the process
\begin{equation*}
    N(z,k):=X(z,k)+(\log (1-z)+1)(k-1) - (\log (1-z)+1) \cdot M_{k-1}(z)
\end{equation*}
is a martingale. Next, consider the difference martingale $L(z,k):=N(z,k)-N(z_0,k)=N(z,k)-X(z_0,k)$ and note that
\begin{equation*}
    \big| L(z,k+1)-L(z,k) \big| \le C|z-z_0|+ C\mathds 1 \{ x_{k+1}\in [z_0,z] \cup [z,z_0] \} +C\mathds 1 \{m_k\ge z_0-n^{-1/2+\epsilon }\}.
\end{equation*}
Thus, by Lemma~\ref{lem:B_1 B_2} we have WVHP 
\begin{equation*}
    V_n:=\sum _{k=1}^n \mathbb E \big[ (L(z,k)-L(z,k-1))^2 \ |  \ \mathcal F _{k-1} \big] \le Cn^{1/2+\epsilon } + C M_n(z_0-n^{-1/2+\epsilon }) \le n^{1/2+2\epsilon }.
\end{equation*}
We obtain using Theorem~\ref{theo:freedman} that
\begin{equation*}
\begin{split}
    \mathbb P \big( \big| L(z,n) \big| \ge n^{1/4+2\epsilon } \big) &\le \mathbb P \big( \big| L(z,n) \big| \ge n^{1/4+2\epsilon } , \ V_n \le n^{1/2+2\epsilon } \big)+\mathbb P \big( V_k \ge n^{1/2+2\epsilon } \big) \\
    & \le C\exp (-n^c).
\end{split}
\end{equation*}
Finally, using the expansion $\log (1-z)+1=-e(z-z_0)+O(n^{-1+2\epsilon })$ we get that WVHP
\begin{equation*}
    \big| X(z,n)-X(z_0,n) -en(z-z_0) \big| \le \big| L(z,k) \big| + Cn^{2\epsilon }+C\cdot |z-z_0| \cdot M_n(z) \le n^{1/4+3\epsilon }.
\end{equation*}
This finishes the proof of the proposition using the discretization trick.
\end{proof}

\subsection{Proof of Theorem~\ref{theo:12}}

The main statement in Theorem~\ref{theo:12} is the scaling limit result
\begin{equation}\label{eq:64}
\Big\{ \Big(  \frac{L_{tn}}{\sqrt{n}} \, , \, \frac{R_{tn}}{\sqrt{n}} \Big)  \Big\}_{t>0} \overset{d}{\longrightarrow } \big\{ \sqrt{2}e^{-1} \big( M_t-B_t, M_t \big) \big\} _{t>0}.
\end{equation}
We start by proving \eqref{eq:64} and then briefly explain how the other statements in Theorem~\ref{theo:12} follow.

By Proposition~\ref{prop:2} and Proposition~\ref{prop:22} we have that 
\begin{equation}\label{eq:41}
     \big\{ n^{-1/2} \big( W(z_0,tn)\ , \ Z(z_0,tn) \big) \big\} _{t>0}  \overset{d}{\longrightarrow } \big\{ \sqrt{2} \big( B_t-\inf _{x<t}B_x\ , \ -\inf _{x<t}B_x   \big) \big\} _{t>0}.
\end{equation}
Moreover, by symmetry 
\begin{equation}\label{eq:42}
      \big\{ \sqrt{2} \big( B_t-\inf _{x<t}B_x\ , \ -\inf _{x<t}B_x   \big) \big\} _{t>0} \overset{d}{=} \big\{ \sqrt{2} \big( M_t-B_t\ , \ M_t   \big) \big\} _{t>0}.
\end{equation}

Next, recall that $L_n=s(z_0,n)$ and therefore, by \eqref{eq:Ws}, WVHP
\begin{equation}\label{eq:43}
    \big| e\cdot L_n- W(z_0,n) \big| \le n^{1/3}.
\end{equation}
The convergence in \eqref{eq:64} clearly follows from \eqref{eq:41}, \eqref{eq:42}, \eqref{eq:43} and the following lemma

\begin{lemma}\label{lem:34}
We have WVHP 
\begin{equation*}
    \big| e\cdot R_n -  Z(z_0,n) \big| \le n^{1/3}.
\end{equation*}
\end{lemma}

\begin{proof}
We write 
\begin{equation}\label{eq:44}
    \big| e\cdot R_n -  Z(z_0,n) \big| \le e\big| R_n-R_n'  \big|+ \big| e \cdot R_n'- M_{n-1}(z_0) \big|+\big| M_{n-1}(z_0)-Z(z_0,n) \big|,
\end{equation}
where $R_n':=\sum _{k=1}^{n-1} \mathds 1 \{ m_k\ge z_0 \} (1-m_k)$ and bound each one of the terms in the right hand side of \eqref{eq:44}.

We start with the first term. Note that $R_n=\big| \{2\le  k\le n : m_k\ge m_{k-1} \ge z_0 \} \big|$ and therefore
\begin{equation*}
    \mathbb E \big[ R_{n+1}-R_n \ |  \ \mathcal F _n  \big] = \mathds 1 \big\{ m_{n}\ge z_0 \big\} (1-m_{n}).
\end{equation*}
It follows that $N_n:=R_n-R_n'$ is a martingale. The predictable quadratic variation of $N_n$ is given by 
\begin{equation*}
    \sum _{k=1}^{n-1} \mathbb E \big[ (N_{k+1}-N_k)^2  \ |  \ \mathcal F _k  \big] \le  \sum _{k=1}^{n-1} \mathds 1  \{ m_{k-1} \ge z_0 \} \le  M_n(z_0) \le n^{1/2+\epsilon },
\end{equation*}
where the last inequality holds WVHP by Lemma~\ref{lem:B_1 B_2}. Thus, using the same arguments as in the proof of Proposition~\ref{prop:3} and by Theorem~\ref{theo:freedman} we have that $|R_n-R_n'|\le n^{1/4+\epsilon }$ WVHP.

We turn to bound the second term in the right hand side of \eqref{eq:44}.  We have 
\begin{equation*}
\begin{split}
    \big| e \cdot R_n'- M_{n-1}(z_0) \big| &\le  \sum _{k=1}^n \mathds 1 \{ m_k\ge z_0 \} \big( 1-e(1-m_k) \big) \\
    &\le C M_n(z_0) n^{-1/4} +M_n\big( z_0+n^{-1/4} \big) \le n^{1/4+2\epsilon },
\end{split}
\end{equation*}
where in the second inequality we separated the sum to $m_k\in [z_0,z_0+n^{-1/4}]$ and $m_k\ge z_0+n^{-1/4}$ and where the last inequality holds WVHP by Lemma~\ref{lem:B_1 B_2}.

The last term in the right hand side of \eqref{eq:44} is bounded by $n^{1/4+\epsilon }$ WVHP using the same arguments as in the first term. Indeed, $Z(z_0,n)-M_{n-1}(z_0)$ is clearly a martingale and its quadratic variation is bounded by $n^{1/2+\epsilon }$ WVHP.
\end{proof}

We turn to prove the other results in Theorem~\ref{theo:12}. By \eqref{eq:64} we have that 
\begin{equation}\label{eq:1729}
    \Big\{ \frac{L_{tn}+R_{tn}}{\sqrt{n}}   \Big\}_{t>0} \overset{d}{\longrightarrow } \big\{  \sqrt{2}e^{-1}(2M_t-B_t) \big\} _{t>0}.
\end{equation}
Recall that a three dimensional Bessel process, $X_t$, is the distance of a three dimensional Brownian motion from the origin. In \cite{Pitman} Pitman proved that 
\begin{equation*}
     \big\{  2M_t-B_t \big\} _{t>0} \overset{d}{=}  \{  X_t \} _{t>0}
\end{equation*}
and therefore the convergence in \eqref{eq:936} follows from \eqref{eq:1729}.
 Next, using \eqref{eq:64} once again we get that  
\begin{equation}\label{eq:cor}
    \frac{L_n}{\sqrt{n}} \overset{d}{\longrightarrow } \sqrt{2}e^{-1}(M_1-B_1), \quad  \quad  \frac{R_n}{\sqrt{n}} \overset{d}{\longrightarrow } \sqrt{2}e^{-1}M_1 
\end{equation}
It is well known (see \cite[Theorem~2.34 and Theorem~2.21]{morters2010brownian}) that $M_1-B_1\overset{d}{=} M_1 \overset{d}{=}|N|$. Substituting these identities into \eqref{eq:cor} finishes the proof of the first two limit laws in \eqref{eq:LR}. 

It remains to prove the last limit law in \eqref{eq:LR}. By \eqref{eq:936} with $t=1$, it suffices to compute the density of $\sqrt{2}e^{-1}X_1$. This is a straightforward computation as $X_1$ is the norm of a three dimensional normal random variable. We omit the details of the computation which lead to the density given in \eqref{eq:symmetric}.

We note that one can compute the density of $2M_1-B_1$ without using the result of Pitman \cite{Pitman}. The following elementary argument was given to us by Iosif Pinelis in a Mathoverflow answer \cite{math}. By the reflection principle  for all $a>b>0$ 
\begin{equation}\label{eq:reflection}
    \mathbb P \big( M_1>a \ , \ B_1<b \big)=\mathbb P \big( B_1\ge 2a-b \big) =1-F(2a-b).
\end{equation}
Equation \eqref{eq:reflection} gives the joint distribution of $M_1$ and $B_1$. From here it is straightforward to compute the joint density of $(M_1,B_1)$ and the density of $2M_1-B_1$. Once again, the details of the computation are omitted.

\subsection{Proof of Theorem~\ref{theo:size}}

Define $E_n(z):=\big|\{ k\le n : x_k\ge z \} \big|$ and consider the martingale 
\begin{equation*}
    N_n:= X(z_0,n) + e E_n(z_0) -n.
\end{equation*}
Theorem~\ref{theo:size} clearly follows from the following two lemmas.

\begin{lemma}\label{lem:5}
We have that 
\begin{equation*}
     \Big\{ \frac{N_{tn}}{\sqrt{n}}  \Big\} _{t>0} \overset{d}{\longrightarrow} \big\{  \sqrt{e-3}\,  B_t \big\}_{t>0}
\end{equation*}
\end{lemma}

\begin{lemma}\label{lem:6}
We have WVHP
\begin{equation*}
    \big| s(1,n) -n/e-N_n/e \big| \le n^{1/3+\epsilon }.
\end{equation*}
\end{lemma}

We start by proving Lemma~\ref{lem:5}.

\begin{proof}[Proof of Lemma~\ref{lem:5}]
On the event $\{m_k\le z_0\}$ we have that 
\begin{equation*}
    N_{k+1}-N_k= \mathds 1 \{ x_{k+1} \le z_0\} \frac{1}{1-x_{k+1}}-\mathds 1 \{x_{k+1}\ge m_k \} \frac{1}{1-m_k}+e\cdot \mathds 1 \{ x_{k+1}\ge z_0 \} -1.
\end{equation*}
Thus, expanding the $10$ terms in $(N_{k+1}-N_k)^2$ we get that on this event
\begin{equation*}
\begin{split}
    &\mathbb E \big[ (N_{k+1}-N_k)^2 \ | \ \mathcal F _k \big] \\
    & = \int _0^{z_0} \frac{dx}{(1-x)^2} -\frac{1}{1-m_k}\int _{m_k}^{z_0} \frac{2\, dx}{1-x} -\int _0^{z_0}\frac{2\, dx}{1-x} +\frac{1}{1-m_k}-\frac{2}{1-m_k}+2+e-2+1 \\
    & =\frac{1}{1-x} \Big|_0^{z_0}+\frac{2\log (1-x)}{1-m_k}\Big| _{m_k}^{z_0}+2\log (1-x) \Big|_{0}^{z_0} -\frac{1}{1-m_k} +1+e\\
    & = \!  e-1-\frac{2}{1-m_k}-\frac{2\log (1-m_k)}{1-m_k}-2 -\frac{1}{1-m_k} +1+e= \frac{-2\log (1-m_k)-3}{1-m_k}+2e-2.
\end{split}
\end{equation*}
Thus, we can write $\mathbb E \big[ (N_{k+1}-N_k)^2 \ | \ \mathcal F _k \big]=f(m_k)$ where $f$ is given by the right hand side of the last equation when $m_k\le z_0$. Moreover, we have that 
\begin{equation*}
    \int _0^{z_0} \frac{f(x)}{1-x}dx=  \frac{-2\log (1-x)-5}{1-x} \Big | _0^{z_0} -(2e-2) \log (1-x) \Big | _0^{z_0}= -3e+5+2e-2=3-e.
\end{equation*}
Thus, by Lemma~\ref{cor:1} we have with very high probability that 
\begin{equation*}
    V_n: = \sum _{k=1}^{n-1} \mathbb E \big[ (N_{k+1}-N_k)^2 \ | \ \mathcal F _k \big]=\sum _{k=1}^{n-1}f(m_k) =(3-e)n+ O(n^{5/6}).
\end{equation*}
This finishes the proof of the lemma using Theorem~\ref{theo:durrett}. 
\end{proof}

We turn to prove Lemma~\ref{lem:6}.

\begin{proof}[Proof of Lemma~\ref{lem:6}]
Recall that  $E_n(z):=\big| \big\{ k\le n : x_k\ge z \big\} \big|$ and let $z_1:=z_0+n^{-1/3}$. By Proposition~\ref{prop:22} the process
\begin{equation*}
    N_k':=X(z_1,k)+(\log (1-z_1)+1)(k-1) - (\log (1-z_1)+1) \cdot M_{k-1}(z_1) +eE_n(z_1)-e(1-z_1)n
\end{equation*}
is a martingale. It is straightforward to check that
\begin{equation*}
    e\cdot s(1,n)-n-N_n=A_1+A_2+A_3+A_4+A_5+A_6
\end{equation*}
where 
\begin{equation*}
\begin{split}
    &A_1:=e\big(s(1,n)-s(z_1,n)-E_n(z_1)\big), \quad A_2:=e\cdot s(z_1,n)-W(z_1,n), \\
    & A_3:= Z(z_1,n), \quad A_4:=N_n'-N_n, \quad A_5:= e(z_0-z_1)n-(\log (1-z_1)+1)(n-1)\\
    &\text{and}\quad A_6=(\log (1-z_1)+1)M_{n-1}(z_1).
\end{split}
\end{equation*}
Next, we bound each of the $A _i$ WVHP. Any uniform point $x\ge z_1$ that arrived before time $n$ such that $x\notin S(1,n)$ can be mapped to the time $k\le n$ in which it was removed. Thus, using also Lemma~\ref{lem:B_1 B_2} we obtain that WVHP 
\begin{equation*}
    0\le E_n(z_1)- \big( s(1,n)-s(z_1,n) \big)  \le M_n(z_1) \le n^{1/3+\epsilon }.
\end{equation*}
This shows that WVHP $|A_1|\le en^{1/3+\epsilon }$.

Using the same arguments as in \eqref{eq:Ws} we get that $|A_2|\le n^{1/3+2\epsilon }$.

By Lemma~\ref{lem:B_1 B_2} we have WVHP $|A_3|\le 3M_n(z_1) \le 3n^{1/3+\epsilon }$.

In order to bound $A_4$ we use the same arguments as in the proof of Proposition~\ref{prop:3}. Define the martingale $L_k:=N_k'-N_k$. We have that 
\begin{equation*}
\big| L_{k+1}-L_k \big|\le Cn^{-1/3} +  C\mathds 1\{x_{k+1}\in [z_0,z_1]\} +C\mathds 1 \{m_k \ge z_0\}.
\end{equation*}
Thus, using Theorem~\ref{theo:freedman} we obtain that WVHP $|N_n-N'_n |=|L_n|\le n^{1/3+3\epsilon }$.

It follows from a second order Taylor expansion of $\log (1-z_1)$ around $z_0$ that $|A_5|\le Cn^{1/3}$. 

Lastly, by Lemma~\ref{lem:B_1 B_2} we have WVHP $|A_6|\le Cn^{-1/3}M_n(z_1)\le Cn^{\epsilon }$.

\end{proof}

\noindent
{\bf Acknowledgment:}\, We thank Ehud Friedgut and Misha
Tsodyks for helpful comments, we thank Ron Peled, Sahar Diskin and Jonathan Zung for fruitful discussions and we thank Iosif Pinelis for proving in \cite{math} that the density function of $2M_1-B_1$ is given by $\frac{2x^2}{\sqrt{2 \pi }} e^{-x^2/2}$.



\begin{thebibliography}{99}
\bibitem{math}
I. Pinelis, What is the distribution of {$2M_1-B_1$} where {$M_t$} is the maximum process of the Brownian motion {$B_t$}. MathOverflow question 409729.
\end{thebibliography}


\end{document}